\documentclass[a4paper]{amsart}
\usepackage{hyperref} 

\usepackage{amsthm}
\usepackage{amsmath,amssymb}
\usepackage{amsxtra}
\usepackage[latin1]{inputenc}
\usepackage[T1]{fontenc}
\usepackage{amssymb,latexsym,array}
\usepackage[UKenglish]{babel}
\usepackage{algorithmicx}
\usepackage[ruled]{algorithm}
\usepackage{algpseudocode}
\usepackage{algc}
\newenvironment{alginc}[1][pseudocode]{\medskip\algsetlanguage{#1}\begin{algorithmic}[0]}{\end{algorithmic}\medskip}
\usepackage[margin=1.1in]{geometry}
\usepackage[lite]{amsrefs}
\usepackage{pstricks}
 \usepackage{epsfig}
 
\usepackage[all]{xy}
\usepackage{latexsym}
\usepackage{enumerate}
\usepackage{prettyref}
\usepackage{hyperref}

\usepackage{mathrsfs}

\newcommand{\op}{\operatorname}

\newcommand{\C}{\mathbb C}

\newcommand{\Z}{\mathbb Z}

\newcommand{\Afour}{\mathcal{A}_4}

\newcommand{\Sfour}{{\mathcal{S}}_4}

\newcommand{\Dtwo}{{\mathcal{D}}_2}
\newcommand{\Dsix}{{\mathcal{D}}_6}
\newcommand{\Dfour}{{\mathcal{D}}_4}
\newcommand{\Dthree}{{\mathcal{D}}_3}

\newcommand{\CtwoSq}{(\Z/2\Z)^2}
\newcommand{\Ctwo}{\Z/2\Z}
\newcommand{\Deight}{{\mathcal{D}}_8}
\newcommand{\CtwoCube}{(\Z/2\Z)^3}
\newcommand{\Cfour}{\Z/4\Z}

\newcommand*{\Homol}{\operatorname{H}}

\renewcommand{\le}{\leqslant}
\renewcommand{\ge}{\geqslant}

\renewcommand{\geq}{\geqslant}

\newcommand{\edgegraph}{
\begin{pspicture}(-0.3,-0.1)(0.7,0.3)
\psdots(-0.2,0.0)
\psdots(0.6,0.0)
\psline(-0.2,0.0)(0.6,0.0)
\end{pspicture} }
\newcommand{\doubleedgegraph}{
\begin{pspicture}(-0.3,-0.1)(1.5,0.3)
\psdots(-0.2,0.0)
\psdots(1.4,0.0)
\psline(-0.2,0.0)(1.4,0.0)
\end{pspicture} }

\newtheorem{theorem}{\bfseries Theorem}
\newtheorem{proposition}[theorem]{\bfseries Proposition}
\newtheorem{corollary}[theorem]{\bfseries Corollary}
\newtheorem{lemma}[theorem]{\bfseries Lemma}

\theoremstyle{definition}
\newtheorem{definition}[theorem]{\bfseries Definition}

\newtheorem{observation}[theorem]{\bfseries Observation}
\newtheorem{remark}[theorem]{\bfseries Remark}

\newtheorem{examplecore}[theorem]{Example}
\newrefformat{ex}{\hyperref[{#1}]{Example~\ref*{#1}}}

\begin{document}

\title[Farrell--Tate and Bredon homology for a modular group of rank 4]
{The Farrell--Tate and Bredon homology for $\op{PSL}_4(\Z)$ \\ via cell subdivisions}

\author[Bui A. T. and A. D. Rahm and M. Wendt]{Anh Tuan Bui, Alexander D. Rahm and Matthias Wendt}
\address{Anh Tuan Bui, 
University of Science - Ho Chi Minh City, Faculty of Math \& Computer Science,
227 Nguyen Van Cu St., District 5, Ho Chi Minh City,
Vietnam \hfill \texttt{batuan@hcmus.edu.vn}
}
\address{Alexander D. Rahm, 
Mathematics Research Unit of Universit\'e du Luxembourg,
Maison du Nombre, 6, Avenue de la Fonte,
L-4364 Esch-sur-Alzette \hfill
\texttt{Alexander.Rahm@uni.lu}
}
\address{Matthias Wendt,
 Mathematisches Institut, Albert-Ludwigs-Universit\"at Freiburg,
 Ernst-Zermelo-Str. 1, 79104 Freiburg im Breisgau, Germany \hfill
\texttt{m.wendt.c@gmail.com} 
}
  
\subjclass[2010]{11F75}
\keywords{Cohomology of arithmetic groups}
\date{\today}

\begin{abstract}
We provide some new computations of Farrell--Tate and Bredon (co)homology for arithmetic groups. 
For calculations of Farrell--Tate or Bredon homology, one needs cell complexes where cell stabilizers fix their cells pointwise. 
We provide two algorithms computing an efficient subdivision of a complex to achieve this rigidity property. 
Applying these algorithms to available cell complexes for $\op{PSL}_4(\Z)$ 
provides computations of Farrell--Tate cohomology for small primes as well as the Bredon homology for the classifying spaces of proper actions with coefficients in the complex representation ring.  
\end{abstract}

\maketitle

\section{Introduction}

Understanding the structure of the cohomology of arithmetic groups is a very important problem with relations to number theory and various K-theoretic areas. 
Explicit cohomology computations usually proceed via the study of the actions of the arithmetic groups on their associated symmetric spaces, 
and recent years have seen several advances in algorithmic computation of equivariant cell structures for these actions. 
To approach computations of Farrell--Tate and Bredon (co)homology of arithmetic groups, one needs cell complexes having a \emph{rigidity property}: 
Cell stabilizers must fix their cells pointwise. 
The known algorithms (using Voronoi decompositions and such techniques, cf. e.g.~\cites{EGS, DGGHSY}) do not provide complexes with this rigidity property,  
and both for the computation of Farrell--Tate cohomology (resp. the torsion at small prime numbers in group cohomology) of arithmetic groups as well as for the computation of Bredon homology,
this lack of rigid cell complexes constitutes a significant bottleneck. 

 In theory, it is always possible to obtain this rigidity property via the barycentric subdivision. 
 However, the barycentric subdivision of an $n$-dimensional cell complex can multiply the number of cells by $(n+1)!$ and thus easily let the memory stack overflow. 
 We provide two algorithms, called Rigid Facets Subdivision, cf. Section~\ref{sec:rfs}, 
 and Virtually Simplicial Subdivision, cf. Section~\ref{Virtually Simplicial Subdivision},
 as well as a combination of them (Hybrid Subdivision, cf. Section~\ref{hybrid subdivision})
 which subdivide cell complexes for arithmetic groups such that stabilizers fix their cells pointwise, 
 but only lead to a controlled increase (in terms of sizes of stabilizer groups) in the number of cells, 
 avoiding an explosion of the data volume. An implementation of the algorithms, cf.~\cite{Bui}, 
 shows that cases like ${\rm PSL}_4(\Z)$ can effectively be treated with it, using commonly available machine resources. 
 For the sake of comparison, barycentric subdivison applied to the cell complex for $\op{PSL}_4(\Z)$ from~\cite{dutour:ellis:schuermann} 
 would produce 3540 times as many top-dimensional cells as Rigid Facets Subdivision does --- see Table~\ref{Numbers of cells for PSL4Z}.

\begin{table}
 \caption{Numbers of cells in the individual dimensions of the studied non-rigid fundamental domain $X$ for PSL$_4(\Z)$ from~\cite{dutour:ellis:schuermann}, its rigid facets subdivsion RFS$(X)$, 
 its virtually simplicial subdivision VSS$(X)$, its hybrid subdivision HyS$(X)$ and its barycentric subdivision BCS$(X)$.}
 \label{Numbers of cells for PSL4Z}
\begin{tabular}{|l|r|r|r|r|r|r|r|}
\hline &&&&&&&\\
Dimension            			&   0 &  1   &     2 &    3   &    4   &    5    &    6\\
\hline &&&&&&&\\
\# Orbits in $X$ 			  & 2 &    2 &     2 &    4   &    4   &    3    &    1\\
\# Cells in $X$     			  &304&  1416&  2040 & 1224   &  332   &  36     &   1\\
\# Orbits in  RFS$(X)$	  		   &17&   213&  1234 & 3025   & 3103   & 1117    &   1\\
\# Cells in  RFS$(X)$  			 &4153& 62592& 268440&472272  &370272  &108096   &  96\\
\# Orbits in  VSS$(X)$  		 &17  &  219 &  1508 & 5082   &  8456  &  6720   & 2040  \\
\#  {Cells in  VSS$(X)$}		 & 4153 & 71952 & 390936 & 974304 & 1238688 & 783360 & 195840 \\
\# Orbits in HyS$(X)$ 			  & 17  &  213  &   1245 &   3095 &    3214 &  1180  &    12\\
\# Cells in HyS$(X)$    		  &4153 &  62592&  271416&  483504&   383520& 114144 &  1152\\
\# Cells in BCS$(X)$ 			 &5353&110352& 644136&1658304 &2138688 & 1359360 & 339840 \\
\hline

\end{tabular}
\end{table}

\subsection{Computations of Farrell--Tate cohomology}
Farrell--Tate cohomology is a modification of cohomology of arithmetic groups which is particularly suitable to investigate torsion related to finite subgroups 
(in particular, the torsion in cohomological degrees above the virtual cohomological dimension). 
While the known cell complexes for arithmetic groups can deal very well with the rational cohomology and torsion at primes which do not divide orders of finite subgroups, 
computations with these complexes run into serious trouble for small prime numbers because the differentials in the relevant spectral sequence are too complicated to evaluate. 
There is a suitable new technique called {torsion subcomplex reduction}, cf.~\cite{accessingFarrell}, 
which produces significantly smaller cell complexes and therefore simplifies the equivariant spectral sequence calculations. To apply this simplification, however, one needs cell complexes with the abovementioned rigidity property. 
Applying Rigid Facets Subdivision to a cell complex for $\op{PSL}_4(\Z)$, 
we have computed the Farrell--Tate cohomology of ${\rm PSL}_4(\Z)$, at the primes $3$ and $5$.
These results can be found in Theorem~\ref{thm:psl4} and Proposition~\ref{5-torsion for PSL4Z}. 
Since the computation proceeds through a complete description of the reduced torsion subcomplex, 
we can compute the torsion above the virtual cohomological dimension in all degrees. 

In the cases which are effectively of rank one ($5$-torsion in $\op{PSL}_4(\Z)$), 
we can check the results of the cohomology computation using torsion subcomplex reduction by comparing to a computation using Brown's formula.

\subsection{Computations of Bredon homology} For any group $G$, Baum and Connes introduced a map from the equivariant $K$-homology of $G$ to the $K$-theory of the reduced $C^*$-algebra of $G$, 
called the assembly map.
For many classes of groups, it has been proven that the assembly map is an isomorphism; and the Baum--Connes conjecture claims that it is an isomorphism for all finitely presented groups $G$ 
(counter-examples have been found only for stronger versions of the Baum--Connes conjecture). The assembly map is known to be injective for arithmetic groups. 
For an overview on the conjecture, see the monograph~\cite{MislinValette}. 

The geometric-topological side of Baum and Connes' assembly map, namely the equivariant $K$-homology, 
can be determined using an Atiyah--Hirzebruch spectral sequence with $E_2$-page given by the Bredon homology $\Homol_n^\mathfrak{Fin}( \underbar{\rm E}G; \thinspace R_\C)$ of the classifying space  
$\underbar{\rm E}G$ for proper actions with coefficients in the complex representation ring $R_\C$ and with respect to the system $\mathfrak{Fin}$ of finite subgroups of $G$. 
This Bredon homology can be computed explicitly, as described by  Sanchez-Garcia~\cites{Sanchez-Garcia, Sanchez-Garcia_Coxeter}.

While for Coxeter groups with a small system of generators~\cite{Sanchez-Garcia_Coxeter} 
and arithmetic groups of rank~$2$~\cite{BianchiGroups}, 
general formulae for the equivariant $K$-homology have been established, 
the only known higher-rank case to date is the example 
$\op{SL}_3(\Z)$ in \cite{Sanchez-Garcia}. 
Although there are by now considerably more arithmetic groups for which cell complexes have been worked out \cites{EGS, dutour:ellis:schuermann, DGGHSY}, 
no further computations of Bredon homology $\Homol_n^\mathfrak{Fin}( \underbar{\rm E}G; \thinspace R_\C)$ 
have been done since 2008 because the relevant cell complexes fail to have the  rigidity property required for Sanchez-Garcia's method. 
We discuss an explicit example, cf. Section~\ref{2PEV}, demonstrating that the rigidity property is essential for the computation of Bredon homology and cannot be circumvented by a different method.

Applying rigid facets subdivision to the cell complex for $\underbar{\rm E}{\rm PSL}_4(\Z)$ from \cite{dutour:ellis:schuermann}, we obtain 
$$
\Homol_n^\mathfrak{Fin}( \underbar{\rm E}{\rm PSL}_4(\Z); \thinspace R_\C) 
\cong \begin{cases}
  0, & n \ge 4,\\
  \Z, & n = 3,\\
  0, & n = 2,\\
  \Z^4, & n = 1,\\
  \Z^{25} \oplus \Z/2, & n = 0.\\
\end{cases}
$$

This is consistent with the rational homology of $\underbar{\rm B}{\rm PSL}_4(\Z)$ as inferred from the results of \cite{dutour:ellis:schuermann}.
As further consistency checks, the authors have paid attention that the homology of the cell complex remains unchanged under our implementation of rigid facets subdivision, 
and that the equivariant Euler characteristic of $\underbar{\rm E}G$ vanishes before and after subdividing.

\subsection*{Organization of the paper.} In Section~\ref{sec:rfs}, we provide the Rigid Facets Subdivision algorithm.
In Section~\ref{Virtually Simplicial Subdivision}, we provide the Virtually Simplicial Subdivision algorithm.
In Section~\ref{Farrell--Tate}, we apply the Rigid Facets Subdivision algorithm to a $\op{PSL}_4(\Z)$, and compare the result with a computation using Brown's conjugacy classes cell complex.
Finally in Section \ref{2PEV}, we provide a counterexample in order to contradict the possibility to compute the Bredon homology from an arbitrary non-rigid cell complex.

\subsection*{Acknowledgements.} 
We are grateful for support by \mbox{Gabor Wiese}'s Fonds National de la Recherche Luxembourg grants (INTER/DFG/FNR/12/10/COMFGREP and AMFOR), 
which did facilitate meetings for this project via visits to Universit\'e du Luxembourg by the first author (thrice for one month) and the third author.
The first author was funded by Vietnam National University - Ho Chi Minh City (VNU-HCM) under grant number C2018-18-02.
We would like to thank Graham Ellis for having supported the development of the first implementation of our algorithms -- the ``Torsion Subcomplexes Subpackage'' for his 
\textsc{Homological Algebra Programming (HAP)} package in GAP.
Further thanks go to \mbox{Sebastian Sch\"onnenbeck}~\cites{BCNS, Sebastian} for having provided us cell complexes for congruence subgroups in $\op{SL}_3(\Z)$, used for benchmarking purposes in this paper,
and especially to Mathieu Dutour Sikiri\'c for having provided the cell complexes for  $\op{SL}_3(\Z)$, Sp$_4(\Z)$ and $\op{PSL}_4(\Z)$ in \textsc{HAP}.
\\
This article is dedicated to the memory of Aled Ellis.

\section{The rigid facets subdivision algorithm} \label{sec:rfs}
In this section, we discuss the rigid facets subdivision algorithm which rigidifies equivariant cell complexes. The core of the method is Algorithm~\ref{RFS}, 
which is expected to run in reasonable time
for input coming from cell complexes for arithmetic groups. The key fact which guarantees that rigid facets subdivision works, is Lemma~\ref{Rigid Facets Lemma} below.

\begin{algorithm} 
\caption{Subdivide to get stabilizers which fix their cells pointwise}
\label{subdivision}
{
\begin{alginc}
\State \bf Input: \rm An $n$-dimensional $\Gamma$-equivariant CW-complex $X$ with finite cell stabilizers and a metric as in Remark~\ref{rem:metric}. 
\State \bf Output: \rm A rigidification of $X$ (that is, an equivalent $\Gamma$-cell complex on which each stabilizer fixes its cell pointwise). 
\State 

\For{$m$ running from $1$ to $n$}
  \For{$\sigma$ running through lifts of $m$-cells in $_\Gamma \backslash X^{(m)}$}
	  \If{$\sigma$ is not rigid}
		  \State Use Algorithm~\ref{RFS} or \ref{VSS} to subdivide $\sigma$ into 
		   a partition $P$, which is a union of rigid $m$-cells,
		   \State disjoint up to boundaries, with a fundamental domain $F$ for the $\Gamma_\sigma$-action on $P$.
		  \State Run through all the $(m+1)$-cells; 
		  if their boundaries contain $\sigma$, 
		  \State \qquad \qquad \qquad \qquad \qquad \qquad \qquad \quad 
		  then replace $\sigma$ by its partition $P$.
		  \State Replace the cell $\sigma$ by $F$ in $_\Gamma \backslash X^{(m)}$.  
	  \EndIf
  \EndFor	
\EndFor
\State

\end{alginc}
}
\end{algorithm}

\begin{definition}
Following the notation in \cite{BuiEllis}, we use the term $\Gamma$-equivariant CW-complex, or simply $\Gamma$-cell complex, to mean a CW-complex $X$ on which a discrete group $\Gamma$ 
acts cellularly, i.e., in such a way that the action induces a permutation of the cells of $X$. 
We say the cell complex is \textit{rigid} if each element in the stabilizer of any cell fixes the cell pointwise.
\end{definition}

\begin{remark}
\label{rem:metric}
The algorithm producing the subdivision of the $\Gamma$-equivariant CW-complex $X$ only modifies combinatorial data, based on the barycentric subdivison of individual cells. We require $X$ to come with a geometric realization, equipped with a metric such that each of the cells of $X$ is convex, 
the restriction of the metric to each cell is CAT(0) and the cell stabilizers act by CAT(0)-isometries. We are not requiring that the metric is CAT(0) on the whole CW-complex. 
Note however that the examples we are most interested in, are those where $\Gamma$ is an arithmetic group and the geometric realization of $X$ is the associated symmetric space.
\end{remark}

\begin{definition}
A \emph{rigidification} $\hat{X}$ of a $\Gamma$-cell complex $X$ is a rigid $\Gamma$-cell complex $\hat{X}$ with the same underlying topological space as $X$.
The map passing through the underlying topological space is then a $\Gamma$-equivariant homeomorphism 
$\hat{X} \to X $ of $\Gamma$-spaces. Note that a $\Gamma$-equivariant homeomorphism of $\Gamma$-spaces does not need to preserve existing cell structure,
so $\hat{X}$ is allowed to have more cells than $X$.
\end{definition}

The outer shell of the rigid facets subdivision is Algorithm \ref{subdivision}, 
which subdivides (whenever necessary) representatives of cell orbits using Algorithm~\ref{RFS}, 
respectively Algorithm~\ref{VSS}.

\begin{proposition} \label{one}
Let $\Gamma$ be a discrete group, and let $X$ be a $\Gamma$-equivariant CW-complex having finitely many $\Gamma$-orbits and finite cell stabilizers. 
Assume furthermore that $X$ is equipped with a metric as in Remark~\ref{rem:metric}. 
Then Algorithm~\ref{subdivision} finds a rigidification of $X$ (with respect to the $\Gamma$-action).  It terminates in finite time. 
\end{proposition}

\begin{proof}
  The key step of the algorithm is proved by Lemma~\ref{Rigid Facets Lemma} below; the rest is a routine induction.  
  Lemma~\ref{Rigid Facets Lemma} is the point where the convexity and isometry requirements are needed. 
  By the finiteness assumptions for orbits and cell stabilizers, 
  the loops are all deterministic over finite index sets. 
  Each operation inside them takes finite time, cf. Corollary~\ref{Algorithm 2 terminates}, whence the claim.
\end{proof}

\begin{observation}
The outer shell Algorithm~\ref{subdivision} can be used with any subdivision algorithm for the cells. 
In particular, replacing the use of Algorithm~\ref{RFS} 
by the barycentric subdivision in Algorithm~\ref{subdivision},
the claims of Proposition~\ref{one} still hold. 
However, as mentioned in the introduction, the reason for developing Algorithm~\ref{RFS} 
is to reduce the blow-up in the number of cells, so as to make the algorithm practically applicable to cell complexes for higher-rank arithmetic groups. 
\end{observation}

We now discuss the actual subdivision to rigidify cells, Algorithm~\ref{RFS}.

\begin{algorithm} [H]
\caption{--- Rigid Facets Subdivision}
\label{RFS}
{
\begin{alginc}

\State \bf Input: \rm An $m$-cell $\sigma$ with stabilizer group $\Gamma_\sigma$, with rigid faces, equipped with a metric as in Remark~\ref{rem:metric}.
\State \bf Output: \rm A $\Gamma_\sigma$-equivariant set of rigid $m$-cells, disjoint up to boundaries,
constituting a partition $P$ of $\sigma$, together with a contractible fundamental domain $F$ (a single $m$-cell) for the action of $\Gamma_\sigma$ on $P$. 
\State 
\begin{itemize}
 \item Sort the $(m-1)$-faces of $\sigma$ into orbits $\{\{g_{jt}\sigma_j\}_t\}_j$ under the action of $\Gamma_\sigma$,
where $j$ is indexing the orbits and $t$ is indexing the cells inside each orbit. 
\item Let $T$ be the list containing the element $g_{11}\sigma_1$.
\end{itemize} 
\While{$\#T < \#\{$orbits of $(m-1)$-faces of $\sigma\}$} 
 \State Choose one cell  $\tau \in \{g_{jt}\sigma_j\}$ in the next orbit $\{g_{jt}\sigma_j\}$ not yet represented in $T$, 
 \If{the union of $\tau$ with the cells in $T$ has vanishing naive Euler characteristic}
    \If{ $\{\tau\} \cup T$ is contractible}
      \State  Construct the boundary $S$ of $\{\tau\} \cup T$.
	\If{$S$ has the naive Euler characteristic of an $(m-2)$-sphere}	
	  \If{ $\#T < \#\{$orbits of $(m-1)$-faces of $\sigma\}$-1} 
	    \State  Add the chosen cell $\tau$ to $T$.
	    \Else
 	      \If{$S$ is simply connected}
	        Add the chosen cell $\tau$ to $T$.
	      \EndIf{}	
	   \EndIf   
	\EndIf
    \EndIf
  \EndIf
\EndWhile
\begin{itemize}
\item Use Algorithm~\ref{envelope} to construct the $m$-cell $F := |\bigcup_{\tau \in T}$ convex envelope($\tau$, barycenter($\sigma$))|. 
\item Let $\Gamma_\sigma^{pw}$ be the subgroup of $\Gamma_\sigma$ which fixes the cell $\sigma$ pointwise.
\item Then $P:=\bigcup_{1 \le t \le |\Gamma_\sigma/\Gamma_\sigma^{pw}|} g_{1t}F$ is the desired partition of $\sigma$.
\item Return $P$ and $F$.
\end{itemize}
\end{alginc}
}
\end{algorithm}

\begin{remark}
A) Essentially,  Algorithm~\ref{RFS} produces a convex union of cells of the barycentric subdivision which is a fundamental domain for the $\Gamma_\sigma$-action. 
The slight complications arise from the fact that we do not actually want to compute the full barycentric subdivision, to gain computational feasibility.
\\
B) The contractibility check on $\{\tau \} \cup T$ in Algorithm~\ref{RFS} can be implemented with Ellis' method~\cite{Hegarty}.
\end{remark}

\begin{lemma}[Rigid Facets Lemma] 
\label{Rigid Facets Lemma}
Let $\sigma$ be a cell (with stabilizer $\Gamma_\sigma$), 
whose faces are all rigid and which is equipped with a metric as in Remark~\ref{rem:metric}. 
Let $\Gamma_\sigma^{pw}$ be the subgroup of $\Gamma_\sigma$ which fixes the cell $\sigma$ pointwise.
Then there is a fundamental domain $F$ for the action of $\Gamma_\sigma/\Gamma_\sigma^{pw}$ on $\sigma$ such that $\sigma$ is tessellated by $|\Gamma_\sigma/\Gamma_\sigma^{pw}|$ copies of $F$.
\end{lemma}

\begin{proof}
  First we have to check that the statement is well defined in the sense that $\Gamma_\sigma/\Gamma_\sigma^{pw}$ is a group. This is the case because for all $g$ in $\Gamma_\sigma$, for all $\gamma$ in $\Gamma_\sigma^{pw}$, for all $x$ in $\sigma$ we have $(g^{-1} \gamma g)x = g^{-1} (\gamma (gx)) =  g^{-1} (gx) = x.$ Therefore, as the kernel of the action of $\Gamma_\sigma$ on $\sigma$, $\Gamma_\sigma^{pw}$ is a normal subgroup; and there is a short exact sequence  of groups,
$$
1 \to \Gamma_\sigma^{pw} \to \Gamma_\sigma \to \Gamma_\sigma/\Gamma_\sigma^{pw} \to 1,
$$
which makes our statement well defined.

Suppose that $\alpha$ is one of the facets of $\sigma$. 
We are going to prove that the size of the orbit of $\alpha$, 
under the action of $\Gamma_\sigma$ on the set of facets of $\sigma$, is $|\Gamma_\sigma/\Gamma_\sigma^{pw}|$. 
Let $\Gamma_\alpha$ be the stabilizer of $\alpha$. 
We claim that $\Gamma_\alpha\cap \Gamma_\sigma=\Gamma_\sigma^{pw}$. 
The action of $\Gamma_\sigma$ on the compact set $closure(\sigma)$ is by homeomorphisms; 
therefore, any element of $\Gamma_\sigma$ fixing $\sigma$ pointwise also fixes the boundary $\partial \sigma$ pointwise. 
Hence $\Gamma_\alpha\cap \Gamma_\sigma\supset\Gamma_\sigma^{pw}$. 
On the other hand, let $g\in \Gamma_\alpha\cap \Gamma_\sigma$. 
Then by assumption on the rigidity of the facets, $g$ fixes the cell $\alpha$ pointwise. 
Since the cell $\sigma$ is convex and the group acts by CAT(0)-isometries, 
the barycenter of~$\sigma$ preserves its distances to the boundary $\partial \sigma$
under the action of $\Gamma_\sigma$ on $\sigma$, and hence remains fixed. 
As a further consequence of the CAT(0) isometry, 
the fixed point set of $g$ extends, by preservation of the distances, 
from the convex envelope of $\alpha$ and the barycenter of $\sigma$ to the whole cell $\sigma$.
Hence, $g$ is an element of $\Gamma_\sigma^{pw}$. Thus, we can conclude that $\Gamma_\alpha\cap\Gamma_\sigma=\Gamma_\sigma^{pw}$. Whence, the size of the orbit of $\alpha$ under the action of $\Gamma_\sigma$ is $|\Gamma_\sigma/\Gamma_\sigma^{pw}|$.

Furthermore, from $\Gamma_\alpha\cap\Gamma_\sigma=\Gamma_\sigma^{pw}$, we see that $\Gamma_\sigma/\Gamma_\sigma^{pw}$ acts freely on the set of facets of $\sigma$. So, we can take one arbitrary representative $\alpha_k$ for each orbit of facets, to unite to a fundamental domain for $\Gamma_\sigma/\Gamma_\sigma^{pw}$ on the set of facets of $\sigma$. Taking the convex envelope $e_k$ of $\alpha_k$ and the barycenter of $\sigma$, we get a fundamental domain $F := \bigcup_k e_k$ for $\Gamma_\sigma/\Gamma_\sigma^{pw}$ on $\sigma$. By the above orbit size calculation, it yields the desired tessellation.
\end{proof}

\begin{corollary} \label{Algorithm 2 terminates}
Algorithm~\ref{RFS} terminates after finitely many steps and produces a rigid subdivision of the cell $\sigma$,
 if the boundary of $\sigma$ admits a contractible fundamental domain $T$ with simply connected boundary.
\end{corollary}
\begin{proof}
The existence of some fundamental domain $F$ for the action $\sigma$ 
is guaranteed by Lemma~\ref{Rigid Facets Lemma}.
The contractibility of the fundamental domain $T$
and the simply connectedness of its boundary $S$ ensure that merging the union of cells 
\begin{center}$\bigcup\limits_{\tau \in T}$ convex envelope($\tau$, barycenter($\sigma$))\end{center}
into one cell $F$ can be realized with the boundary construction 
$\partial F =  \bigcup_{\tau \in T} \tau \cup \bigcup_{s \in S} e(s),$
where 
$e(s)$ is the convex envelope of $s$ and the barycenter of $\sigma$.
\end{proof}

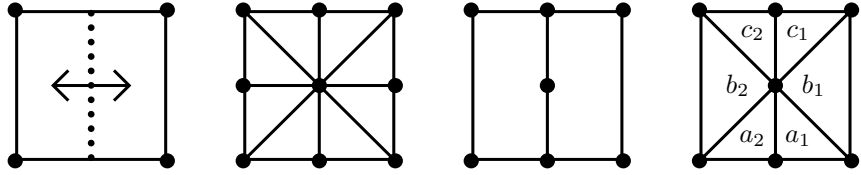
\begin{figure}
\caption{From left to right: 1) A full square with natural cell structure and action by its stabilizer mirroring it onto itself.
2) Barycentric subdivision of the square. 3)~Rigid facet subdivision of the square with respect to the mirroring action.
\mbox{4) Virtually Simplicial Subdivision} of the square with respect to the mirroring action; 
note that the fundamental domain $F$ can be chosen arbitrarily by selecting one cell among each of the pairs 
$\{ a_1, a_2\}$, $\{ b_1, b_2\}$ and $\{ c_1, c_2\}$.}
\label{comparison example}
\scalebox{1} 
{
\begin{pspicture}(-0.2,-0.2)(11.2,2.2)
\psframe[linewidth=0.04,dimen=outer](0,0)(2,2)
\psdots[dotsize=0.2](0,0)
\psdots[dotsize=0.2](0,2)
\psdots[dotsize=0.2](2,0)
\psdots[dotsize=0.2](2,2)
\psline[linewidth=0.09cm,linestyle=dotted](1,0)(1,2)
\psline[linewidth=0.04](0.7,0.8)(0.5,1)(0.7,1.2)
\psline[linewidth=0.04cm](0.5,1)(1.5,1)
\psline[linewidth=0.04](1.3,0.8)(1.5,1)(1.3,1.2)
\psframe[linewidth=0.04,dimen=outer](3,0)(5,2)
\psdots[dotsize=0.2](3,0)
\psdots[dotsize=0.2](5,2)
\psdots[dotsize=0.2](3,2)
\psdots[dotsize=0.2](5,0)
\psline[linewidth=0.04cm](4,0)(4,2)
\psdots[dotsize=0.2](4,0)
\psdots[dotsize=0.2](4,1)
\psdots[dotsize=0.2](4,2)
\psline[linewidth=0.04cm](3,1)(5,1)
\psdots[dotsize=0.2](3,1)
\psdots[dotsize=0.2](5,1)
\psline[linewidth=0.04cm](3,0)(5,2)
\psline[linewidth=0.04cm](5,0)(3,2)
\psframe[linewidth=0.04,dimen=outer](6,0)(8,2)
\psdots[dotsize=0.2](6,0)
\psdots[dotsize=0.2](8,2)
\psdots[dotsize=0.2](8,0)
\psdots[dotsize=0.2](6,2)
\psline[linewidth=0.04cm](7,0)(7,2)
\psdots[dotsize=0.2](7,0)
\psdots[dotsize=0.2](7,1)
\psdots[dotsize=0.2](7,2)
\psframe[linewidth=0.04,dimen=outer](9,0)(11,2)
\psdots[dotsize=0.2](9,0)
\psdots[dotsize=0.2](11,2)
\psdots[dotsize=0.2](9,2)
\psdots[dotsize=0.2](11,0)
\psline[linewidth=0.04cm](10,0)(10,2)
\psdots[dotsize=0.2](10,0)
\psdots[dotsize=0.2](10,1)
\psdots[dotsize=0.2](10,2)
\psline[linewidth=0.04cm](9,0)(11,2)
\psline[linewidth=0.04cm](11,0)(9,2)
\rput(9.7,1.7){$c_2$}
\rput(10.3,1.7){$c_1$}
\rput(9.5,1.0){$b_2$}
\rput(9.7,0.3){$a_2$}
\rput(10.5,1.0){$b_1$}
\rput(10.3,0.3){$a_1$}
\end{pspicture}
}
\end{figure}

\begin{algorithm} 
  \caption{Constructing the union of the convex envelopes}
  \label{envelope}
	{
	  \begin{alginc}
	    
	    \State \bf Input: \rm A list $T$ of $j$-cells, connected by adjacency,
	    \State \qquad \qquad		and an $m$-cell $\sigma$ such that all $\tau \in T$ are faces of $\sigma$.
	    \State \bf Output: \rm A $(j+1)$-cell $F$ which has the same underlying topological space as 
	     \begin{center}
		$\bigcup_{\tau \in T}$convex envelope($\tau$, barycenter($\sigma$)).
	     \end{center}
			\State
			\State Record the barycenter of $\sigma$ as a new vertex, with stabilizer $\Gamma_\sigma$.
			\State Enumerate the finite set $S := \{ \rho \medspace {\rm is} \medspace (j-1){\text{-cell}}  \medspace |  \medspace \exists !  \medspace \tau \in T : \rho \in \partial \tau\}$.
			\State \textit{Then $S$ contains all the $(j-1)$-faces $\rho$ of all $\tau \in T$ such that $\rho$ is not a common face of any two $j$-cells $\tau_1, \tau_2 \in T$. }
			\begin{itemize}

			\item For each $s \in S$, take the convex envelope $e(s)$ of $s$ and the barycenter of $\sigma$.
			\item Record $e(s)$ as an oriented $j$-cell, 
			 with boundary
			$$\partial e(s) = \{s\} \cup \bigcup_{\varepsilon \in \partial s} \{\text{convex envelope}(\varepsilon,{\rm barycenter}(\sigma))\}.$$
			For the stabilizers, we record $\Gamma_{e(s)} = \Gamma_s \cap \Gamma_\sigma$ and  
			$\Gamma_{\text{convex envelope}(\varepsilon,{\rm barycenter}(\sigma))} = \Gamma_\varepsilon \cap \Gamma_\sigma$.
			
			\medskip
			
			\textit{The boundary of the $(j+1)$-cell $F$ consists of all the $j$-faces $\tau \in T$ and $e(s)$ (for all $s \in S$).
			Here, we have to take care of which of the newly constructed cells $e(s)$ are on the same $\Gamma_\sigma$-orbit.
			In order to decide this, we make use of their common vertex, the barycenter of $\sigma$ : }
			\item Identify the orbits $e(s_1)$ and $e(s_2)$ if and only if $\exists \medspace \gamma \in \Gamma_\sigma : \gamma s_1 = s_2$.
			\item Attribute arbitrary orientations to $\Gamma_\sigma$-representatives of the new cells $e(s)$, 
			  and spread them on their $\Gamma_\sigma$-orbit using the above identifications.
			\item Return the $(j+1)$-cell $F$ with boundary
			$$\partial F = \bigcup_{\tau \in T} \tau \cup \bigcup_{s \in S} e(s) $$
			and stabilizer $\Gamma_\sigma \cap \bigcap_{\tau \in T} \Gamma_\tau$, where the orbits are subject to the above specified identifications.
			\end{itemize}
		\end{alginc}
	}
\end{algorithm}

\section{Virtually Simplicial Subdivision}\label{Virtually Simplicial Subdivision}
As a compromise between Rigid Facets Subdivision (RFS) and Barycentric Subdivision (BCS), we can make a simplification of RFS, 
which we shall call Virtually Simplicial Subdivision (VSS) and detail in Algorithm~\ref{VSS} below.
For a cell complex $X$, VSS$(X)$ is a refinement of RFS$(X)$, and BCS$(X)$ is a further refinement of VSS$(X)$,
as illustrated in Figure~\ref{comparison example}.
\begin{algorithm} [H]
\caption{--- Virtually Simplicial Subdivision}
\label{VSS}
{
\begin{alginc}

\State \bf Input: \rm An $m$-cell $\sigma$ with stabilizer group $\Gamma_\sigma$, with rigid faces, equipped with a metric as in Remark~\ref{rem:metric}.
\State \bf Output: \rm A $\Gamma_\sigma$-equivariant set of rigid $m$-cells, disjoint up to boundaries,
consituting a partition $P$ of $\sigma$, together with a fundamental domain $F$ (a set of $m$-cells) for the action of $\Gamma_\sigma$ on $P$. 
\State 
\begin{itemize}
 \item Sort the $(m-1)$-faces of $\sigma$ into orbits $\{\{g_{jt}\sigma_j\}_t\}_j$ under the action of $\Gamma_\sigma$,
where $j$ is indexing the orbits and $t$ is indexing the cells inside each orbit. 
\item Use Algorithm~\ref{envelope} to construct the $m$-cells $F_j := $ convex envelope($g_{j1}\sigma_j$, barycenter($\sigma$)), where $j$ runs through all orbits.
\item Let $\Gamma_\sigma^{pw}$ be the subgroup of $\Gamma_\sigma$ which fixes the cell $\sigma$ pointwise.
\item Then $P:=\bigcup_j \bigcup_{1 \le t \le |\Gamma_\sigma/\Gamma_\sigma^{pw}|} g_{jt}F_j$ is the desired partition of $\sigma$.
\item Return $P$ and the set $F := \{F_j\}_j$.
\end{itemize}
\end{alginc}
}
\end{algorithm}

\begin{corollary} \label{VSS terminates}
Algorithm~\ref{VSS} 
terminates after finitely many steps and produces a rigid subdivision of the cell $\sigma$.
\end{corollary}
\begin{proof}
The existence of the fundamental domain $F$ for the action $\sigma$ 
is guaranteed by Lemma~\ref{Rigid Facets Lemma}.
We do not need any topological properties for $F$, because its cells already have the structure constructed with Algorithm~\ref{envelope}.
\end{proof}

\section{Hybrid Subdivision} \label{hybrid subdivision}
The bulk of the processing time of Rigid Facets Subdivision (RFS) and Virtually Simplicial Subdivision (VSS) is being spent on the top-dimensional cells.
We can save a considerable amount of processing time as compared to RFS if we do not merge the top-dimensional cells. 
And we can do that by applying RFS until we are in (top minus 1) dimensions, and then achieve the subdivision using VSS on the top-dimensional cells.
We will call this process ``Hybrid Subdivision (HyS)''.
It is faster than both RFS and VSS on large cell complexes, and it produces considerably less cells than VSS.
We note however that the VSS process on the top-dimensional cells adds a few low-dimensional cells when connecting to the barycenter, 
so for instance in Table~\ref{Numbers of cells for Sp4Z}, the low-dimensional numbers of cells are slightly higher for HyS than for RFS, rather than being equal.

\section{Comparison of the various subdivision algorithms}
As illustrated in Figure \ref{comparison example}, 
all of the cells constructed by Rigid Facets Subdivision are also constructed by barycentric subdivision;
and usually, Rigid Facets Subdivision is coarser than barycentric subdivision. 
However, the worst-case complexity of Rigid Facets Subdivision is not better than that of barycentric subdivision: 
If $X$ is a single $n$-simplex with the natural permutation action of the symmetric group $\Sigma_{n+1}$ acting on the vertices, 
then any rigidification will need to produce at least $(n+1)!=\#\Sigma_{n+1}$ top-dimensional cells for $X$. 
However, the point is that the average cell complex for interesting arithmetic groups has many of its cells already almost rigid and, only very few with maximally possible stabilizer. 
Therefore, the complexity for the cases of interest is significantly better than that of the barycentric subdivision, as evidenced by Table \ref{Numbers of cells for PSL4Z}.
As we see from Figure~\ref{comparison example}, 
Rigid Facets Subdivision is a minimal rigidification only for the top-dimensional cells;
in lower dimensions it creates some superfluous cells at the barycenters of facets,
and cells connecting those barycenters with the constructed fundamental domain for the boundary of the subdivided cell. 

In Table \ref{Numbers of cells for PSL4Z}, note that we do not have the numbers of orbits for the barycentric subdivision, because the latter was not constructed;
only the numbers of cells have been calculated.
In such a construction, the boundaries of the $4$-cells would span a $1658304 \times 2138688$-matrix, and even though that matrix is quite sparse, 
the authors have not tried to store that amount of information on a machine. 
In retrospective however, considering that the amount of cells that have been constructed with Virtually Simplicial Subdivision
comes close to the amount of cells expected for barycentric subdivision, it turns out that this issue would not prevent applying barycentric subdivision
(which in the end was not done, because Table~\ref{runtime} was interpreted by the authors as an exponential growth of the runtime).

In contrast, for SL$_3(\Z)$, the barycentric subdivision is almost reached (see Table \ref{Numbers of cells for SL3Z}) :
In dimensions $0$ and $1$, the same cell structure is obtained with both subdivision methods,
the only difference being that the four tetrahedra which represent the  top-dimensional cells, are merged into one polyhedron when passing to the rigid facets subdivision,
along three triangles of trivial stabilizers.
Note also that Soul\'e did in~\cite{Soule} carry out the barycentric subdivision by hand.
This means for the $2$-torsion subcomplex of SL$_3(\Z)$, which we will discuss in Section~\ref{SL3Z at the prime 2},
that we can extract it equivalently by hand from~\cite{Soule}, or by machine after rigid facets subdivision.

In order to benchmark the run-time, in Table~\ref{runtime}, we denote the subgroup of $\op{SL}_3(\mathbb{Z})$ consisting of
\begin{enumerate}
 \item all matrices whose first column agrees with
the first standard basis vector modulo $2$ by $\Gamma_1$;
\item  all matrices which are upper triangular modulo $2$ by $\Gamma_2$.
\end{enumerate}

\begin{table}
\caption{Time spent on a single processor for subdividing available cell complexes.}
\label{runtime}
\begin{tabular}{|l|r|r|r|r|r|}
\hline &&&&&\\
Arithmetic group	&$\op{SL}_3(\mathbb{Z})$ &   $\Gamma_1$ in $\op{SL}_3(\mathbb{Z})$& $\Gamma_2$ in $\op{SL}_3(\mathbb{Z})$ & Sp$_4(\Z)$ & $\op{PSL}_4(\mathbb{Z})$ \\
\hline &&&&&\\
Rigid Facets Subdivision         &8s            &  17s          &	 26s  & 	  24s&	69 minutes\\
Virtually Simplicial Subdivision &9s            &  22s 		&	 33s  &		  43s&  111 minutes\\
Hybrid Subdivision		 &9s		&  20s		&	 31s  & 	  33s&  33 minutes \\
Barycentric Subdivision          &13s           &  56s		&	104s  &          296s&	? \\
\hline
\end{tabular}
\end{table}

\section{Example: Farrell--Tate cohomology of \texorpdfstring{$\op{SL}_3(\mathbb{Z})$}{SL3Z at the prime 2}} 
\label{SL3Z at the prime 2}
 \begin{table}
 \caption{Numbers of cells in the individual dimensions, of the studied non-rigid fundamental domain $X$ for SL$_3(\Z)$, 
 its rigid facets subdivsion RFS$(X)$, its virtually simplicial subdivision VSS$(X)$, its hybrid subdivision HyS$(X)$ and its barycentric subdivision BCS$(X)$ 
 (the latter three coincide for SL$_3(\Z)$).}
 \label{Numbers of cells for SL3Z}
\begin{tabular}{|l|r|r|r|r|}
\hline &&&&\\
Dimension            			&   0 &  1   &     2 &    3   \\
\hline &&&&\\
\# Orbits in $X$ 			  & 1  & 1 &  2 &  1\\
\# Cells in $X$     			 &  16 & 24 & 10 & 1\\
\# Orbits in  RFS$(X)$	  		  & 5 &  11 &  8 &  1\\
\# Cells in  RFS$(X)$  			 & 51 & 194 & 168  & 24\\ 
\# Orbits in VSS$(X)$, HyS$(X)$ and BCS$(X)$			&   5 &  11 &  11 &  4 \\ 
\# Cells in VSS$(X)$, HyS$(X)$ and BCS$(X)$ 			 &  51 & 194 & 240 & 96\\
\hline
\end{tabular}
\end{table}
 
From the non-rigid cell complex describing the action of SL$_3(\mathbb{Z})$ on its symmetric space, provided by Mathieu Dutour Sikiri\'c in \cite{HAP},
we obtain, after applying Rigid Facets Subdivision, the following $2$-torsion subcomplex, in accordance with Soul\'e's subdivision~\cite{Soule}.
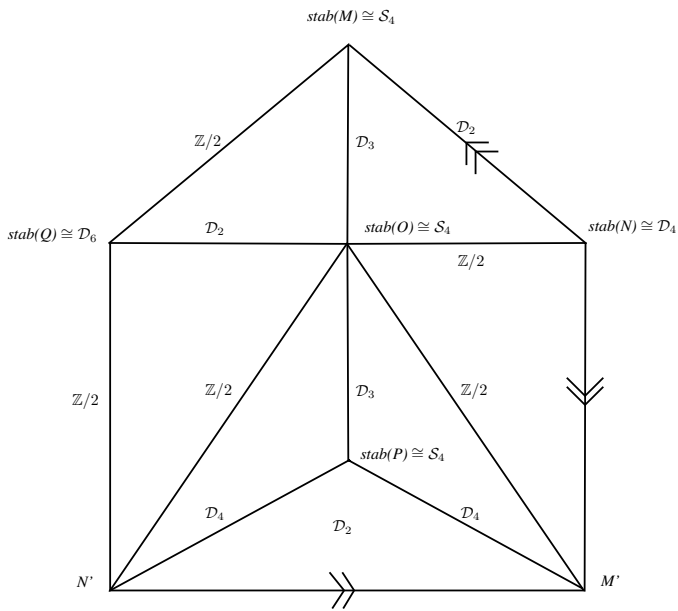
\begin{center} 
\scalebox{0.6} 
{
\begin{pspicture}(-1.3,-7.44125)(11.894688,7.46125)
\pstriangle[linewidth=0.04,dimen=outer](5.8803124,-6.42125)(10.46,7.72)
\psline[linewidth=0.04](11.050312,-6.36125)(5.9103127,-3.52125)(5.8903127,1.25875)(0.6903125,1.27875)(0.6903125,-6.40125)(5.9303126,-3.52125)(5.9503126,-3.54125)
\psline[linewidth=0.04](11.090313,-6.42125)(11.110312,1.27875)(5.8903127,1.25875)(5.9103127,5.65875)(0.6703125,1.27875)(0.6903125,1.25875)
\psline[linewidth=0.04](5.9303126,5.65875)(11.110312,1.27875)(11.130313,1.25875)
\usefont{T1}{ptm}{m}{it}
\rput(5.9759374,6.26875){stab(M) $\cong \Sfour$}
\usefont{T1}{ptm}{m}{it}
\rput(-0.575,1.44875){stab(Q) $\cong \Dsix$}
\usefont{T1}{ptm}{m}{it}
\rput(7.217656,1.60875){stab(O) $\cong \Sfour$}
\usefont{T1}{ptm}{m}{it}
\uput[0](11.0,1.62875){stab(N) $\cong \Dfour$}
\usefont{T1}{ptm}{m}{it}
\uput[0](6.0,-3.4){stab(P) $\cong \Sfour$}
\usefont{T1}{ptm}{m}{it}
\rput(0.14765625,-6.21125){N'}
\usefont{T1}{ptm}{m}{it}
\rput(11.657657,-6.19125){M'}
\uput[90](8.5,3.5){$\Dtwo$} 
\uput[0](5.9,3.5){$\Dthree$} 
\uput[0](5.9,-2.0){$\Dthree$} 
\uput[0](5.4,-5.0){$\Dtwo$} 
\uput[180](3.3,3.5){$\Z/2$} 
\rput(0.14765625,-2.21125){$\Z/2$} 
\uput[0](2.6,-2.0){$\Z/2$} 
\uput[0](2.6,-4.7){$\Dfour$} 
\uput[0](8.2,-2.0){$\Z/2$} 
\uput[0](8.2,-4.7){$\Dfour$} 
\uput[0](2.6,1.6){$\Dtwo$} 
\uput[270](8.6,1.2){$\Z/2$} 
\psline[linewidth=0.04](8.99,3.5)(8.5,3.5)(8.5,3.0)
\psline[linewidth=0.04](9.2,3.3)(8.7,3.3)(8.7,2.8)
\psline[linewidth=0.04](10.7,-1.7)(11.1,-2.1)(11.5,-1.7)
\psline[linewidth=0.04](10.7,-1.9)(11.1,-2.3)(11.5,-1.9)
\psline[linewidth=0.04](5.49875,-6.0)(5.81875,-6.4)(5.55875,-6.8)
\psline[linewidth=0.04](5.77875,-6.0)(6.03875,-6.4)(5.79875,-6.8)
\end{pspicture} 
}
\end{center} 
Here, the three edges $NM$, $NM'$ and $N'M'$ have to be identified as indicated by the arrows.
All of the seven triangles belong with their interior to the $2$-torsion subcomplex, 
each with stabilizer $\Z/2$, except for the one which is marked to have stabilizer $\Dtwo$.
Using {torsion subcomplex reduction} (cf.~\cite{accessingFarrell}), 
we reduce this subcomplex to

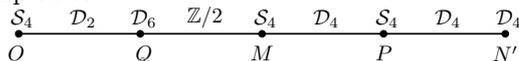
\begin{center}
 \scalebox{0.8} 
{
\begin{pspicture}(-1.9,-0.9)(8.5,0.3)
        \psdots(-0.0,0.0)
        \psline(-0.0,0.0)(2.0,0.0)
        \uput{0.1}[90](-0.0,0.0){ $\Sfour$}
                \uput{0.4}[270](-0.1,0.2){ $O$}
        \psdots(2,0.0)
                \uput{0.1}[90](1.0,0.0){ $\Dtwo$}
        \uput{0.1}[90](2.0,0.0){ $\Dsix$}
                \uput{0.4}[270](2.0,0.2){ $Q$}
        \psline(2,0.0)(8,0.0)
                        \uput{0.1}[90](3.0,0.0){ $\Z/2$}
        \uput{0.1}[90](4.0,0.0){ $\Sfour$}
                \uput{0.4}[270](4.0,0.2){$M$}
        \psdots(4,0.0)
                        \uput{0.1}[90](5.0,0.0){ $\Dfour$}
               \uput{0.1}[90](6.0,0.0){ $\Sfour$}
                  \uput{0.2}[270](6.0,0.0){$P$}
                          \psdots(6,0.0)
                      \uput{0.1}[90](7.0,0.0){ $\Dfour$}
        \uput{0.1}[90](8.0,0.0){ $\Dfour$}
                \uput{0.4}[270](8.0,0.2){$N'$}
        \psdots(8,0.0)
\end{pspicture} 
}
\end{center}
and then to
\begin{center}
 \scalebox{0.8} 
{
\begin{pspicture}(-1.9,-0.2)(8.0,0.3)
        \uput{0.1}[90](2.0,0.0){ $\Sfour$}
        \psdots(2,0.0)
        \psline(2,0.0)(6,0.0)
                        \uput{0.1}[90](3.0,0.0){ $\Z/2$}
        \uput{0.1}[90](4.0,0.0){ $\Sfour$}
        \psdots(4,0.0)
                        \uput{0.1}[90](5.0,0.0){ $\Dfour$}
               \uput{0.1}[90](6.0,0.0){ $\Sfour$}
        \psdots(6,0.0)
\end{pspicture} 
}
\end{center}
which is the geometric realization of Soul\'e's diagram of cell stabilizers. 
This yields the mod $2$ Farrell cohomology as specified in~\cite{Soule}.

\bigskip

Also for $\op{PSL}_4(\mathbb{Z})$, Rigid Facets Subdivision allows us to obtain the $2$-torsion subcomplex,
but because of the complexity the latter (see Table~\ref{non-reduced 2-torsion subcomplex}), 
the current implementation of Torsion Subcomplex Reduction does not get it down to manageable size.
In upcoming joint work, the authors are going to present a new implementation of Torsion Subcomplex Reduction,
which will overcome this problem using Discrete Morse Theory.

\begin{table} 
 \caption{Numbers of cells in the $2$-torsion subcomplex for PSL$_4(\Z)$ before reduction, 
 obtained after rigid facets subdivsion of the studied cell complex, sorted into isomorphism types of their stabilizers.
Here, $G_1$ and $G_2$ are non-trivial group extensions 
$1 \to \Afour \times \Afour \to G_1 \to \Ctwo \to 1$
and $1 \to G_0 \to G_2 \to \Ctwo \to 1$
for $1 \to (\Z/2\Z)^4 \to G_0 \to \Z/3\Z \to 1$}
\label{non-reduced 2-torsion subcomplex}
$\begin{array}{|l|r|r|r|r|r|r|r|r|r|}
\hline &&&&&&&&&\\
\text{Stabilizer type} & \Afour &       G_1 & \CtwoSq& \Sfour &\Ctwo&\Deight & \CtwoCube &        G_2         & \Cfour \\
\hline &&&&&&&&&\\
\text{Vertices}        & 2      &       1   &     5  &  4     &  1  &   2    &       1   &         1          &  0  \\
\text{Edges}           & 2      &       0   &    24  &  2     &101  &   5    &       1   &         0          &  5  \\
\text{2-cells}         & 0      &       0   &    27  &  0     &326  &   0    &       0   &         0          &  4  \\
\text{3-cells}         & 0      &       0   &     8  &  0     &340  &   0    &       0   &         0          &  0  \\
\text{4-cells}         & 0      &       0   &     0  &  0     &116  &   0    &       0   &         0          &  0  \\
\hline
\end{array}$
\end{table}

\begin{table}
 \caption{Numbers of cells in the individual dimensions, of the non-rigid fundamental domain $X$ for Sp$_4(\Z)$ described in~\cite{MacPhersonMcConnell} and implemented by Dutor Sikiri\'{c}~\cite{HAP}, 
 its rigid facets subdivsion RFS$(X)$, its virtually simplicial subdivision VSS$(X)$, its hybrid subdivision HyS$(X)$ and its barycentric subdivision BCS$(X)$.}
 \label{Numbers of cells for Sp4Z} 
\begin{tabular}{|r|c|c|c|c|c|}
\hline &&&&&\\
Dimension            			&   0 &  1   &     2 &    3   &    4   \\
\hline &&&&&\\
In $X$: \# Orbits | \# Cells  & 2 | 76 & 2 | 216 & 3 | 180 & 3 | 40 & 1 | 1 \\
In  RFS$(X)$: \# Orbits | \# Cells & 8 | 185 & 31 | 800 & 54 | 1048 & 30 | 448    & 1 | 16    \\ 
In VSS$(X)$: \# Orbits | \# Cells & 8 | 185 & 39 | 1152 & 106 | 2536 & 122 | 2352 & 49 | 784   \\
In  HyS$(X)$: \# Orbits | \# Cells & 8 | 185& 35 | 864 & 72 | 1336 & 53 | 816  & 10 | 160  \\
In BCS$(X)$ : \# Orbits | \# Cells  &11 | 513 & 90 | 3488 & 295 | 7904 & 368 | 7392 & 154 | 2464   \\
\hline
\end{tabular} \normalsize
\end{table}

\section{Example: Farrell--Tate cohomology of \texorpdfstring{$\op{Sp}_4(\mathbb{Z})$}{Sp4Z} at odd primes}
\label{sp4}

We shortly discuss the results of applying rigid facets subdivision and torsion subcomplex reduction (for odd primes) to the non-rigid cell complex describing the action of $\op{Sp}_4(\mathbb{Z})$ on the associated symmetric space $\op{Sp}_4(\mathbb{R})/\op{U}(2)$, described in \cite{MacPhersonMcConnell} and implemented by Mathieu Dutour Sikiri\'c in \cite{HAP}. 

The $3$-torsion subcomplex for $\op{Sp}_4(\mathbb{Z})$ is a single vertex whose stabilizer group is the group $[72,30]$ in the GAP \verb!SmallGroup! library. This group is of the form $\op{C}_3\times ((\op{C}_6\times\op{C}_2)\rtimes \op{C}_2)$ and its mod 3 cohomology is isomorphic to the mod 3 cohomology of $\op{C}_3\times \op{S}_3$. In particular, we get isomorphisms
\[
\op{H}^i(\op{Sp}_4(\mathbb{Z}),\mathbb{F}_3) \cong \op{H}^i(\op{C}_3\times\op{S}_3,\mathbb{F}_3), \quad i>\op{vcd}(\op{Sp}_4(\mathbb{Z})).
\]
The Hilbert--Poincar{\'e} series for the mod 3 cohomology of $\op{C}_3\times\op{S}_3$ is given by 
\[
\op{HP}_{\op{C}_3\times\op{S}_3}(T;3)=\frac{(1+T)(1+T^3)}{(1-T^2)(1-T^4)}.
\]
In \cite{brownstein:lee}*{(6.7)}, the Hilbert--Poincar{\'e}-series for $\op{Sp}_4(\mathbb{Z})$ has been computed. In the notation of loc.cit., the difference to the Hilbert--Poincar{\'e}-series for $\op{C}_3\times \op{S}_3$ above is given by
\[
P(\Gamma^6_0)-P(B'\rtimes\mathbb{Z}/2)=\frac{1+T^3+T^4+T^5}{1-T^4}-\frac{1+T+T^3+T^4}{1-T^4}=\frac{T^5-T}{1-T^4}=-T.
\]
In particular, the Hilbert--Poincar{\'e} series for $\op{C}_3\times\op{S}_3$ and $\op{Sp}_4(\mathbb{Z})$ agree above the virtual cohomological dimension of $\op{Sp}_4(\mathbb{Z})$ and thus the computation via rigid facets subdivision and torsion subcomplex reduction agrees with the computation of Brownstein and Lee in \cite{brownstein:lee}. 

The $5$-torsion subcomplex for $\op{Sp}_4(\mathbb{Z})$ is a single vertex whose stabilizer group is $\op{C}_{10}$, whose mod 5 cohomology is isomorphic to the mod 5 cohomology of $\op{C}_5$. This agrees directly with \cite{brownstein:lee}*{Corollary 6.3} in degrees above the virtual cohomological dimension. 

Table~\ref{Numbers of cells for Sp4Z} provides an overview of the numbers of cells resp. orbits in the complexes subdivided by rigid facets subdivision and barycentric subdivision, respectively.

\section{Example: Farrell--Tate cohomology of \texorpdfstring{$\op{PSL}_4(\mathbb{Z})$}{PSL4Z} at the prime 3} 
\label{Farrell--Tate}

Applying the rigid facets subdivision algorithm to the $\op{PSL}_4(\Z)$-equivariant cell complex from~\cite{dutour:ellis:schuermann},
extracting the $3$-torsion subcomplex, and reducing it using the methods of~\cite{accessingFarrell}, 
we get the following graph of groups $\mathcal{T}$, 
decorated with the groups stabilizing the cells that are the pre-images of the projection to the quotient space.

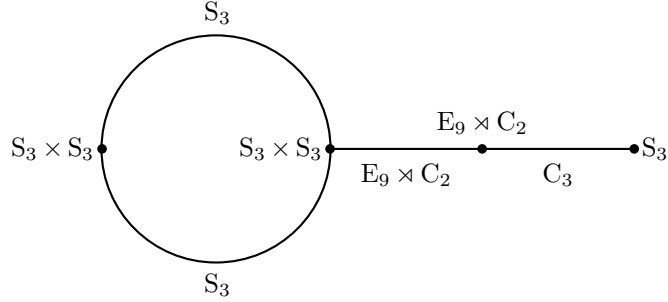
\begin{center}
\begin{pspicture}(-0.1,-0.1)(8,4)
      \pscircle(2,2){1.52}
      \psdots(3.5,2)
      \psdots(0.5,2)      
      \uput{0}[180](2.2,3.8){$\op{S}_3$}
      \uput{0}[180](2.2,0.2){$\op{S}_3$}
      \uput{0}[180](0.4,2){$\op{S}_3 \times \op{S}_3$}
      \uput{0}[180](3.4,2){$\op{S}_3 \times \op{S}_3$}
      \psline(3.5,2)(7.5,2)
       \uput{0.2}[270](4.5,2){$\op{E}_9\rtimes\op{C}_2$}
	\uput{0.2}[270](6.5,2){$\op{C}_3$}
      \psdots(7.5,2)
       \uput{0.1}[0](7.5,2){$\op{S}_3$}
      \psdots(5.5,2)      
       \uput{0.2}[90](5.5,2){$\op{E}_9\rtimes\op{C}_2$}
\end{pspicture} 
\end{center}
Here, $\op{E}_9\rtimes\op{C}_2$ denotes the semi-direct product of the elementary abelian group of order $9$ with $\op{C}_2$, where $\op{C}_2$ acts by inversion (corresponding to the group [18,4] in GAP's \verb!SmallGroups! library). The machine computation provided the following system of morphisms among the above cell stabilizers. The $\op{E}_9 \rtimes \op{C}_2$ edge stabilizer admits an isomorphism of groups $\phi$ (not the identity, though) to the $\op{E}_9 \rtimes \op{C}_2$ vertex stabilizer and an inclusion into the $\op{S}_3 \times \op{S}_3$ vertex stabilizer. Of the two $\op{S}_3$ edge stabilizers, one has maps $\op{diag}(1,1)$ and $\op{diag}(1,0)$ to the two $\op{S}_3 \times \op{S}_3$ vertex stabilizers, and the other one has maps $\op{diag}(1,-1)$ and  $\op{diag}(-1,-1)$ to the two $\op{S}_3 \times \op{S}_3$ vertex stabilizers. The $\op{C}_3$ edge stabilizer admits an inclusion into the $\op{E}_9 \rtimes \op{C}_2$ vertex stabilizer, and an inclusion into the $\op{S}_3$ vertex stabilizer.

By the properties of torsion subcomplex reduction, the 
$\op{PSL}_4(\Z)$-equivariant cohomology of the $3$-torsion subcomplex is isomorphic to the 
$\op{PSL}_4(\Z)$-equivariant cohomology of the above graph of groups $\mathcal{T}$. 
Similarly, the Farrell--Tate cohomology of $\op{PSL}_4(\Z)$ at the prime $3$ is isomorphic to the Farrell--Tate cohomology of the above graph of groups. In the following, we evaluate the isotropy spectral sequence 
\[
E_1^{p,q}=\bigoplus_{\sigma\in \mathcal{T}_p}\op{H}^q(\op{Stab}(\sigma);\mathbb{F}_3)\Rightarrow \widehat{\op{H}}^{p+q}(\op{PSL}_4(\mathbb{Z});\mathbb{F}_3) 
\]
converging to Farrell--Tate cohomology. As we only consider a graph, the spectral sequence is concentrated in the two columns $p=0,1$. The differential $\op{d}_1$ is induced from the inclusions of subgroups, up to the sign coming from the choice of orientation of the graph. Since the spectral sequence is only concentrated in the first two columns, we will have $E_2=E_\infty$. Since we are interested in field coefficients, there are no extension problems to solve at the $E_\infty$-page.
 
The relevant cohomology groups of the finite groups are:
\begin{itemize}
\item $\op{H}^\bullet(\op{C}_3;\mathbb{F}_3)\cong\mathbb{F}_3[x](a)$ with $\deg a=1$ and $\deg x=2$.
\item $\op{H}^\bullet(\op{S}_3;\mathbb{F}_3)\cong \mathbb{F}_3[y](b)$ with $\deg b=3$ and $\deg y=4$.
\item By the K\"unneth formula, $\op{H}^\bullet(\op{S}_3\times\op{S}_3;\mathbb{F}_3)\cong \op{H}^\bullet(\op{S}_3;\mathbb{F}_3)^{\otimes 2}$. 
\item By the Hochschild--Serre spectral sequence, $\op{H}^\bullet(\op{E}_9\rtimes\op{C}_2;\mathbb{F}_3)\cong \mathbb{F}_3[x_1,x_2](a_1,a_2)^{\op{C}_2}$ where $\deg x_i=2$, $\deg a_i=1$ and  $\op{C}_2$ acts by multiplication with $-1$ on all the generators.
\end{itemize}
To describe the $\op{d}_1$-differential, it is enough to note that the restriction map associated to the inclusion $\op{C}_3\hookrightarrow\op{S}_3$ is the inclusion of $\op{C}_2$-invariants.

Now, for the evaluation of the spectral sequence, we first deal with the edges attached to the loop. 
\begin{enumerate}
\item
The restriction map 
\[
\phi\oplus (\op{Res}^{\op{S}_3}_{\op{C}_3} \circ \op{pr}_2^\ast) : 
\op{H}^\bullet(\op{E}_9\rtimes\op{C}_2;\mathbb{F}_3)\to 
\op{H}^\bullet(\op{E}_9\rtimes\op{C}_2;\mathbb{F}_3)\oplus 
\op{H}^\bullet(\op{C}_3;\mathbb{F}_3)
\]
is injective, and the cokernel is isomorphic to $\op{H}^\bullet(\op{C}_3;\mathbb{F}_3)$. Here $\phi$ denotes the isomorphism $\op{E}_9\rtimes\op{C}_2\cong\op{E}_9\rtimes\op{C}_2$ appearing as stabilizer inclusion in the reduced torsion subcomplex.
\item
The inclusion of the dihedral vertex group into the cyclic edge group is an injection
\[
\op{Res}^{\op{S}_3}_{\op{C}_3} : \op{H}^\bullet(\op{S}_3;\mathbb{F}_3)\hookrightarrow \op{H}^\bullet(\op{C}_3;\mathbb{F}_3)
\]
given by the inclusion of the invariant elements for the $\op{C}_2$-action by $-1$. Therefore, the cokernel is concentrated in degrees $1,2\bmod 4$ (except for the degree $0$).
\end{enumerate}

Therefore, we can reduce the $E_1$-page of the spectral sequence as follows: from (1), we find that we can remove the two summands for $\op{E}_9\rtimes\op{C}_2$ from the columns $p=0$ and $p=1$, respectively; but in turn, we have to replace the restriction map for $\op{E}_9\rtimes\op{C}_2\hookrightarrow \op{S}_3\times\op{S}_3$ by a map from the cohomology of $\op{S}_3\times\op{S}_3$ to the cokernel of the restriction map in (2). However, since the latter is concentrated in degrees 1 and 2, the induced restriction map is trivial in the generating degrees $3$ and $4$, hence it is trivial. Therefore, the $E_1$-page decomposes: one contribution comes from the cokernel of the restriction map $\op{H}^\bullet(\op{S}_3;\mathbb{F}_3)\hookrightarrow \op{H}^\bullet(\op{C}_3;\mathbb{F}_3)$, the other contribution comes from the loop connecting the two copies of $\op{S}_3\times\op{S}_3$ via the $\op{S}_3$-edges. 

The cohomology of the loop is computed as follows:

\begin{enumerate}
 \item[(3)]
The restriction maps $\op{diag}(1,1), \op{diag}(1,0), \op{diag}(1,-1)$ and $\op{diag}(-1,-1) :$
\[
 \mathbb{F}_3[x_4,y_4](a_3,b_3)\cong \op{H}^\bullet(\op{S}_3\times\op{S}_3;\mathbb{F}_3)\to
\op{H}^\bullet(\op{S}_3;\mathbb{F}_3)\cong\mathbb{F}_3[z_4](c_3)
\]
are surjective. 
On the kernel of $\op{diag}(1,1)$, the restriction of $\op{diag}(1,-1)$ is still surjective.
On the kernel of $\op{diag}(-1,-1)$, the restriction of $\op{diag}(1,0)$ is still surjective.
\end{enumerate}

Therefore, the cohomology of one edge of the loop is killed already by the restriction map from any one of the vertex groups. 
Number the $\op{S}_3\times\op{S}_3$-vertices by $1$ and $2$, 
and the $\op{S}_3$-edges by $a$ and $b$. 
The restriction from the vertex group $1$ to the edge $a$ is surjective. 
Removing this part from the spectral sequence, 
the restriction from the vertex group $2$ to the edge $a$ is trivial, 
but we still have the restriction to the edge $b$. 
This kills the edge cohomology $b$, 
showing that the differential $d_1$ is surjective in the loop part of the $E_1$-page. 
The kernel of the differential consists then exactly of two copies of the kernel of a restriction map 
$\op{Res}^{\op{S}_3\times\op{S}_3}_{\op{S}_3}$. 

\begin{theorem}
\label{thm:psl4}
The Farrell--Tate cohomology of the group $\op{PSL}_4(\mathbb{Z})$
(with coefficients in $\mathbb{F}_3$) \emph{in degrees $\geq 2$} is given as follows:
$$
\widehat{\op{H}}^\bullet(\op{PSL}_4(\mathbb{Z});\mathbb{F}_3)
\cong 
\left(\ker\op{Res}^{\op{S}_3\times\op{S}_3}_{\op{S}_3} \right)^{\oplus 2} 
\oplus \op{coker}\left(\op{H}^{\bullet-1}(\op{S}_3;\mathbb{F}_3)\to\op{H}^{\bullet-1}(\op{C}_3;\mathbb{F}_3)\right) . 
$$
This, in particular also computes the $3$-torsion group cohomology of $\op{PSL}_4(\mathbb{Z})$ above the virtual cohomological dimension. 

The cokernel of the $d_1$-differential in degree $0$ 
and hence the first cohomology $\widehat{\op{H}}^1(\op{PSL}_4(\mathbb{Z});\mathbb{F}_3)$ 
is of $\mathbb{F}_3$-rank $1$, coming from the loop of the $3$-torsion graph.
\end{theorem}

\begin{remark}
The kernel comes from the $p=0$ column of the $E_2=E_\infty$-page. The cokernel comes from the $p=1$ column and consequently has a shift. The cup product on the kernel is the one induced from cohomology of $\op{S}_3\times\op{S}_3$, and cup-product with classes in the cokernel is always $0$.
\end{remark}

Inspired by Grunewald, we consider the Hilbert--Poincar{\'e} series of the Farrell--Tate cohomology of $\op{PSL}_4(\mathbb{Z})$ with coefficients in $\mathbb{F}_\ell$ :
$$ \op{HP}_{\op{PSL}_4(\mathbb{Z})}(T;\ell) := \sum\limits_{q = 1}^{\infty} \dim \widehat{\op{H}}^q(\op{PSL}_4(\mathbb{Z});\mathbb{F}_\ell) \cdot T^q.$$
\begin{corollary}
 The Hilbert--Poincar{\'e} series of the $3$-torsion Farrell--Tate cohomology of 
 $\op{PSL}_4(\mathbb{Z})$ (for degrees $\geq 1$) is then 
\[
\op{HP}_{\op{PSL}_4(\mathbb{Z})}(T;3)=T +\frac{2(T^3+T^4+T^6+T^7)}{(1-T^4)^2}+\frac{T^2+T^3}{1-T^4}.
\]
\end{corollary}
\begin{remark}
Note that the above calculation describes the Farrell--Tate cohomology in all degrees, not just some small ones. Essentially, the computer calculation produces the reduced torsion subcomplex (which encodes the cohomology for all degrees). The spectral sequence is evaluated using the cup-product structure. Note that the finiteness results for group homology imply that the cup-product structure for both group and Farrell--Tate cohomology is finitely generated. Using suitable commutative algebra packages, such computations of the ring structure (and therefore additive computations for all cohomological degrees) could probably also be automated. 
\end{remark}

We make the following consideration on the compatibility of our result for Farrell--Tate cohomology
with the result of Dutour--Ellis--Sch\"urmann~\cite{dutour:ellis:schuermann} 
for group homology in low degrees. 
The isomorphism types computed in the latter article are to correspond as follows to 
the evaluation of our above Hilbert-Poincar\'e series in those degrees.
\begin{center}
$
\op{H}_q(\op{PSL}_4(\mathbb{Z});\mathbb{Z}) \cong
$\scriptsize $
\begin{cases}
 0, & q = 1,\\
 (\mathbb{Z}/2)^3, & q = 2,\\
 \mathbb{Z} \oplus (\mathbb{Z}/4)^2 \oplus (\mathbb{Z}/3)^2 \oplus \mathbb{Z}/5, & q = 3,\\
 (\mathbb{Z}/2)^4 \oplus \mathbb{Z}/5, & q = 4,\\
 (\mathbb{Z}/2)^{13}, & q = 5,\\
\end{cases}
$
\hfill
\normalsize
---
\hfill 
$
\dim \widehat{\op{H}}^q(\op{PSL}_4(\mathbb{Z});\mathbb{F}_3) =
$\scriptsize $
\begin{cases}
 1,& q = 1,\\
  1,& q = 2,\\
   3,& q = 3,\\
    2,& q = 4,\\
     0,& q = 5.\\
\end{cases}
$ \normalsize
\end{center}
For this to be consistent, the Farrell--Tate cohomology groups in degrees $1$ and $2$ need to vanish in group homology; 
so, these should be annihilated by differentials from the orbit space. 
We have evidence for this in degree $1$, 
since the loop in the graph becomes contractible in the orbit space of the full locally symmetric space.  
In degree $3$, one of the summands in $\op{H}_3(\op{PSL}_4(\mathbb{Z});\mathbb{Z})$
is rationally non-trivial and must come from the orbit space.
This means that only the submodule $(\mathbb{Z}/3)^2$ can come from Farrell--Tate cohomology,
and the third dimension that we observe in degree~$3$ Farrell--Tate cohomology
must belong to the degree~$2$ stabilizer cohomology class that is annihilated by the above mentioned differentials from the orbit space.

From Theorem~\ref{thm:psl4}, we deduce that the degree~$2$ Farrell--Tate class can only 
come from $$\op{coker}\left(\op{H}^{\bullet-1}(\op{S}_3;\mathbb{F}_3)\to\op{H}^{\bullet-1}(\op{C}_3;\mathbb{F}_3)\right).$$
Then, this class and its group homology counterpart sit at position $p = 1,$ $q = 1$ 
in the respective equivariant spectral sequence, 
and hence the annihilating differential, emanating from the orbit space homology module 
$\mathbb{Z} \subset \op{H}_3(\op{PSL}_4(\mathbb{Z});\mathbb{Z})$ 
sitting at position $p = 3,$ $q = 0$, must be of second degree.

In degrees~$4$ and~$5$, the dimensions already agree via the Universal Coefficient Theorem, 
so here we infer that the submodule $(\mathbb{Z}/3)^2$ in degree~$3$ 
should actually come from Farrell--Tate cohomology,
so it should be stabilizer cohomology that is not hit by higher degree differentials. 

\section{Homological 5-torsion in \texorpdfstring{$\op{PSL}_4(\mathbb{Z})$}{PSL4Z}}
\label{sec:psl4}
We applied the rigid facets subdivision algorithm to the $\op{PSL}_4(\mathbb{Z})$-equivariant cell complex of~\cite{dutour:ellis:schuermann}, extracted the $5$-torsion subcomplex, and reduced it using the methods of~\cite{accessingFarrell} to the following graph $\mathcal{T}$:
\scalebox{0.7}
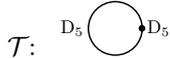
{
\begin{pspicture}(-1,-0.5)(1,0.5)
      \pscircle(0,0){0.52}
      \psdots(0.5,0)
      \uput{0.1}[180](-0.5,0){$\op{D}_5$}
      \uput{0.1}[0](0.5,0){$\op{D}_5$}
 \end{pspicture} 
 }
The $d^1$-differential of the equivariant spectral sequence on  $\mathcal{T}$ is zero, because the isomorphisms at edge end and edge origin cancel each other. 
Then the $E_1 = E_\infty$ page is concentrated in the columns $p = 0$ and $1$, with dimensions over $\mathbb{F}_5$ being $1$ in rows $q$ congruent to $3$ or $4$ mod $4$, and zero otherwise. This yields
\begin{proposition} \label{5-torsion for PSL4Z}
We observe on Farrell cohomology: 
$\dim_{\mathbb{F}_5} \widehat{\op{H}}^{p+q}(\op{PSL}_4(\mathbb{Z});\mathbb{F}_5) = $
\scriptsize $
\begin{cases}
                                                                                        1, &  p + q \equiv 1\mod 4, \\
                                                                                        0, &  p + q \equiv 2 \mod 4, \\
                                                                                        1, &  p + q \equiv 3 \mod 4, \\
                                                                                        2, &  p + q \equiv 4 \mod 4.                                                                                      
                                                                                       \end{cases}
$\normalsize                                  
\end{proposition}
We check this result with a computation of $\widehat{\op{H}}^\bullet(\op{PSL}_4(\mathbb{Z});\mathbb{F}_5)$ using Brown's complex~\cite{Brown}*{last chapters}. 
In this case, it is standard that the set $\{1,\zeta_5,\zeta_2^3,\zeta_5^3\}$ is an integral basis of $\mathcal{O}_{\mathbb{Q}(\zeta_5)}$ and in particular $\mathbb{Z}[\zeta_5]=\mathcal{O}_{\mathbb{Q}(\zeta_5)}$ is a Dedekind ring. 

We can therefore use Reiner's result \cite{reiner:1955} to determine conjugacy classes of $\op{C}_5$-subgroups in $\op{GL}_4(\mathbb{Z})$. Since both $\mathbb{Z}$ and $\mathbb{Z}[\zeta_5]$ have trivial class group, there is only one isomorphism class of $\mathbb{Z}[\op{C}_5]$-module with nontrivial action and $\mathbb{Z}$-rank $4$. Hence, there is a unique conjugacy class of cyclic order $5$ subgroup in $\op{GL}_4(\mathbb{Z})$. Since the center of $\op{GL}_4(\mathbb{Z})$ is of order $2$, the same is true for $\op{PGL}_4(\mathbb{Z})$. 

Now there is a necessary modification to deal with the case $\op{SL}_4(\mathbb{Z})$, along the lines of the discussion in \cite{sl2ff}. While conjugacy classes of $\op{C}_5$-subgroups in $\op{GL}_4(\mathbb{Z})$ correspond to isomorphism classes of $\mathbb{Z}[\op{C}_5]$-modules, the conjugacy classes of $\op{C}_5$-subgroups in $\op{SL}_4(\mathbb{Z})$ correspond to such modules \emph{equipped with an additional orientation}, i.e., a choice of isomorphism $\det M\cong \bigwedge^4_{\mathbb{Z}} M\cong \mathbb{Z}$. The conjugacy class  in $\op{GL}_4(\mathbb{Z})$ lifts to $\op{SL}_4(\mathbb{Z})$, and the corresponding module has two different choices of orientation. The Galois group $\op{Gal}(\mathbb{Q}(\zeta_5)/\mathbb{Q})\cong\mathbb{Z}/4\mathbb{Z}$ acts on the set of oriented modules. The action exchanges the orientations. Therefore, there is one conjugacy class of $\op{C}_5$-subgroup in $\op{SL}_4(\mathbb{Z})$ stabilized by $\mathbb{Z}/2\mathbb{Z}\hookrightarrow \op{Gal}(\mathbb{Q}(\zeta_5)/\mathbb{Q})$.

The centralizer of this $\op{C}_5$-subgroup is the group of norm-1 units of $\mathbb{Z}[\zeta_5]$, which by Dirichlet's unit theorem is isomorphic to 
\[
\ker\left(\mathbb{Z}[\zeta_5]^\times\to\mathbb{Z}^\times\right)\cong
\mathbb{Z}/10\mathbb{Z}\times\mathbb{Z}.
\]
As in \cite{sl2ff}*{Section 5}, the normalizer is an extension of the centralizer by an action of the stabilizer of the corresponding oriented module in the Galois group. We noted above that the Galois group 
$\mathbb{Z}/4\mathbb{Z}$ exchanges the two orientations of the trivial module, hence the stabilizer is the subgroup $\mathbb{Z}/2\mathbb{Z}\subset \mathbb{Z}/4\mathbb{Z}$. 
The normalizer therefore is of the form 
$
\left(\mathbb{Z}/10\mathbb{Z}\times \mathbb{Z}\right)\rtimes\mathbb{Z}/2\mathbb{Z}.
$
The action of $\mathbb{Z}/2\mathbb{Z}$ on $\mathbb{Z}/10\mathbb{Z}$ is by multiplication with $-1$ because the action of the Galois group is via the identification $\mathbb{Z}/4\mathbb{Z}\cong\mathbb{Z}/5\mathbb{Z}^\times$. The action of $\mathbb{Z}/2\mathbb{Z}$ on $\mathbb{Z}$ is trivial: the full Galois group acts on $\mathbb{Z}$ via a surjective homomorphism $\mathbb{Z}/4\mathbb{Z}\to\mathbb{Z}^\times\cong\mathbb{Z}/2\mathbb{Z}$. The stabilizer of the oriented module in the Galois group lies in the kernel of the above action, as claimed. Therefore, the normalizer is in fact of the form $\op{D}_{10}\times\mathbb{Z}$. 

Applying the formulas from \cite{sl2ff}*{Section 3}, the Farrell--Tate cohomology of the normalizer is of the form $\mathbb{F}_5[a_2^{\pm 2}](b_1^3)^{\oplus 2}\oplus \mathbb{F}_5[a_2^{\pm 2}](b_1^3)_{-1}^{\oplus 2}$ where the lower subscript $-1$ indicates a degree shift by $-1$. The Hilbert--Poincar{\'e} series for the positive degrees is 
$
\frac{T^3+2T^4+T^5}{1-T^4}=\frac{T^3(1+T)^2}{1-T^4}.
$

The computations in \cite{dutour:ellis:schuermann} show that the $5$-torsion in integral homology of $\op{PSL}_4(\mathbb{Z})$ of dimension 1 in degrees $0,3\bmod 4$ and trivial otherwise. By the universal coefficient theorem, this agrees with the above computation.

\section{Necessity of Rigidity for Bredon homology}
\label{2PEV}

From a non-rigid cell complex, i.e., a cell complex where cell stabilizers do not necessarily fix the corresponding cell pointwise, one can compute classical group homology via the equivariant spectral sequence with coefficients in the orientation module. Such an orientation module, where elements of the stabilizer group act by multiplication with $1$ or $-1$, depending on whether they preserve or reverse the orientation of the cell, cannot exist for Bredon homology. We make this precise in the following statement:

\begin{proposition}
 There is no module-wise variation of the Bredon module with coefficients in the complex representation ring and with respect to the system of finite subgroups
 such that Bredon homology can be computed from a non-rigid cell complex.
\end{proposition}
 
 \begin{proof}
 We provide a counterexample in order to contradict the possibility to compute the Bredon homology from an arbitrary non-rigid cell complex.
 
Consider the classical modular group $\op{PSL}_2(\Z)$. A model for ${\rm \underbar{E}PSL}_2(\Z)$ is given by the modular tree~\cite{trees}. There is a rigid cell complex structure $T_1$ on it, given as follows.
By~\cite{trees}, the modular tree admits a strict fundamental domain for $\op{PSL}_2(\Z)$, of the shape 
$$\Z/3\Z \edgegraph \Z/2\Z$$
with vertex stabilizers as indicated and trivial edge stabilizer. We obtain the cell complex $T_1$ by tessellating the modular tree with the $\op{PSL}_2(\Z)$-images of this fundamental domain. Obviously, $T_1$ is rigid, and it yields the Bredon chain complex
$$ 
0 \to R_\C(\langle 1 \rangle) \stackrel{d}{\longrightarrow} R_\C(\Z/2\Z) \oplus R_\C(\Z/3\Z) \to 0.
$$
The map $d$ in the above Bredon chain complex is injective, and as $R_\C(\Z/n\Z) \cong \Z^n$, we read off
$$ 
\Homol_1^\mathfrak{Fin}( \underbar{\rm E}{\rm PSL}_2(\Z); \thinspace R_\C) = 0,
\qquad
\Homol_0^\mathfrak{Fin}( \underbar{\rm E}{\rm PSL}_2(\Z); \thinspace R_\C) \cong \Z^4.
$$    

Now we equip the modular tree with an alternative equivariant cell structure $T_2$, induced by the non-strict fundamental domain 
$$
\Z/3\Z \doubleedgegraph \Z/3\Z
$$
where the (set-wise) edge stabilizer is $\Z/2\Z$, flipping the edge onto itself. It can be seen as a ramified double cover of the fundamental domain for $T_1$ discussed above. A system of representative cells for $T_2$ is given by the edge of double length, and one vertex of stabilizer type $\Z/3\Z$. This yields a chain complex
$$ 
0 \to \widetilde{R_\C(\Z/2\Z)} \to R_\C(\Z/3\Z) \to 0,
$$
where the tilde could be any construction which takes the non-trivial $\Z/2\Z$-action on the edge of double length into account (similar to the coefficients in the orientation module for group homology computed from non-rigid cell complexes). But no matter how this construction is done, from $R_\C(\Z/3\Z) \cong \Z^3$, we can never reach $\Homol_0^\mathfrak{Fin}( \underbar{\rm E}{\rm PSL}_2(\Z); \thinspace R_\C) \cong \Z^4.$ Hence $T_2$ is our desired counterexample.
\end{proof}

\begin{remark}
 We could of course drop the condition ``module-wise'' in the above proposition, and investigate
 whether there is a reasonable construction which maps the representation ring to a complex  of modules and yields a quasi-isomorphism from the total complex to the Bredon complex for the subdivided tree. 
 But with such a construction, one would only superficially hide the fact that
 one needs to know how to subdivide in order to get the constructed complexes right. 
 This means that it will not be practicable to compute Bredon homology with respect to the system of finite subgroups and coefficients in the complex representation ring 
 without subdividing the cell complex under consideration to make it rigid.
\end{remark}

\end{document}